\documentclass[11pt,a4paper]{amsart}
\usepackage{a4wide}

\usepackage[utf8]{inputenc}
\usepackage[english]{babel}

\usepackage{amsfonts,amssymb,amsmath, enumitem}

\usepackage[nobysame,alphabetic, initials]{amsrefs}

\usepackage{xcolor}
\usepackage[colorlinks,
    linkcolor={red!50!black},
    citecolor={blue!50!black},
    urlcolor={blue!80!black}]{hyperref}

\newcommand{\bsm}{\left(\begin{smallmatrix}}
\newcommand{\esm}{\end{smallmatrix}\right)}
\usepackage{tikz}
\usepackage[all]{xy}
\usetikzlibrary{calc, matrix, arrows, cd}

\newtheorem{theorem}{Theorem}[section]

\newtheorem{corollary}[theorem]{Corollary}

\newtheorem{proposition}[theorem]{Proposition}

\theoremstyle{definition}

\newtheorem{remark}[theorem]{Remark}

\newtheorem*{claim*}{Claim}

\newcommand{\Z}{\mathbb{Z}}
\newcommand{\Q}{\mathbb{Q}}
\newcommand{\R}{\mathbb{R}}
\newcommand{\C}{\mathbb{C}}

\newcommand{\Id}{\operatorname{Id}}

\newcommand{\nd}{\operatorname{nd}}

\newcommand{\ks}{\operatorname{ks}}
\newcommand{\Aut}{\operatorname{Aut}}
\newcommand{\coker}{\operatorname{coker}}

\def\ol{\overline}
\newcommand{\unaryminus}{\scalebox{0.75}[1.0]{\( - \)}}

\newcommand{\BTOPSpin}{\operatorname{BTOPSpin}}

\newcommand{\BSTOP}{\operatorname{BSTOP}}

\begin{document}
\title{Locally flat simple spheres in~$\C P^2$}
\begin{abstract}
The fundamental group of the complement of a locally flat surface in a $4$-manifold is called the knot group of the surface.
In this article we prove that two locally flat $2$-spheres in~$\C P^2$ with knot group $\Z_2$ are ambiently isotopic if they are homologous.
This combines with work of Tristram and Lee-Wilczy\'{n}ski,  as well as the classification of~$\Z$-surfaces, to complete a proof of the statement:
a class~$d \in H_2(\C P^2) \cong \Z$ is represented by a locally flat sphere with abelian knot group if and only if~$|d| \in \lbrace 0,1,2\rbrace$; and this sphere is unique up to ambient isotopy.
\end{abstract}

\author[A.~Conway]{Anthony Conway}
\address{The University of Texas at Austin, Austin TX 78712}
\email{anthony.conway@austin.utexas.edu}
\author[P.~Orson]{Patrick Orson}
\address{California Polytechnic State University, San Luis Obispo, CA 93407}
\email{patrickorson@gmail.com}

\maketitle

\section{Introduction}

We call a locally flat surface~$F \subseteq \C P^2$ with~$\pi_1(\C P^2 \setminus F)$ abelian (and therefore cyclic)~ \emph{simple}. In particular, we call a locally flat sphere with~$\pi_1(\C P^2 \setminus F)=\Z_2$, a \emph{$\Z_2$-sphere}. Throughout the article submanifolds are assumed to be locally flat, not necessarily moreover smooth.

\medbreak

The main theorem of this article is the following.
\begin{theorem}
\label{thm:Main}
Homologous $\Z_2$-spheres in $\C P^2$ are ambiently isotopic.
\end{theorem}

With this, we deduce the following result about the classification of
simple spheres in~$\C P^2$. 
Theorem~\ref{thm:Main} is the final component of the result, the other components already being present in the literature, coming from various sources. 
Here, and throughout the paper, we fix an identification $H_2(\C P^2)\cong \Z$, by fixing $[\C P^1]$ as the generator.

\begin{theorem}
\label{thm:Classification}
A class $d\in H_2(\C P^2)$ is represented by a
simple sphere if and only if~$d \in \lbrace 0,\pm1,\pm2\rbrace$, and in these cases the simple sphere is unique up to ambient isotopy.
\end{theorem}
\begin{proof}
The class~$d \in H_2(\C P^2)$ is represented by a
simple sphere if and only if~$|d| \in \lbrace 0,1,2 \rbrace$: the class~$d=0$ is represented by any smoothly embedded sphere that bounds a smoothly embedded~$3$-ball,
the classes~$d=1$ is represented by~$\C P^1 \subseteq\C P^2$, the class~$d=2$ is represented by the degree two curve~$x^2+y^2+z^2=0$.
By a van Kampen argument in the cases~$d=0,1$, or a Zariski-van Kampen argument in the case $d=2$ (see~\cite{MR1506962}), the spheres are seen to be simple.
Reversing the orientation of these surfaces leads to representatives for the classes~$d=-1$ and~$d=-2$. 
Work of Tristram~\cite[Page~264]{Tristram} implies that these are the only classes represented by locally flat spheres.

If~$F \subseteq \C P^2$ is a
simple sphere representing $d \in H_2(\C P^2)$, then a Mayer-Vietoris argument confirms its exterior~$X:=\C P^2 \setminus \nu(F)$ must have fundamental group~$\pi_1(X) \cong H_1(X)\cong \Z_{|d|}$.
Let~$F \subseteq \C P^2$ be a 
simple sphere with~$[F]=d \in H_2(\C P^2)$.
For~$d=0,\pm 1$,  such a sphere is known to be unique up to topological isotopy; for~$d=\pm 1$ this follows from~\cite[Theorem~1.2]{LeeWilczy}; for~$d=0$ the result can be deduced from~\cite[Theorem~1.4]{ConwayPowell} (see Proposition~\ref{prop:0Class}).
For~$d=\pm 2$, this follows from Theorem~\ref{thm:Main}.
\end{proof}

\subsection{Previous uniqueness results for closed,  oriented, simple surfaces}

Let~$M$ be a closed,  oriented, simply-connected~$4$-manifold.
Given a nonzero, divisibility $d$ class $x \in H_2(M)$,  Lee and Wilczy\'{n}ski prove ambient isotopy uniqueness results for locally flat spheres in~$M$ representing $x$ for: the case~$d=1$; the case where
\begin{equation}\label{eq:inequality}
b_2(M) > \underset{0 \leq j <d}{\operatorname{max}} \ \Big| \sigma(M)-\frac{2j(d-j)}{d^2}x\cdot x \Big|
\end{equation}
and $b_2(M)>6$~\cite[Theorem 1.1]{LeeWilczy}; and for some cases of nonzero even classes when~$3 \leq b_2(M) \leq 6$, 
and~\eqref{eq:inequality} is satisfied
~\cite[Addendum 1]{LeeWilczy}. Note that if~\eqref{eq:inequality} is satisfied, then~$M$ is not a definite $4$-manifold.
These results build on their earlier work which focused on odd~$d$~\cite[Theorem~1.2]{LeeWilczyOdd}.
Hambleton-Kreck obtained similar results~\cite[Theorem~4.5]{HambletonKreckCancellation}.
Boyer also studied the case~$d=1$, without any genus restrictions~\cite[Theorem~F]{BoyerRealization}.
Sunukjian has proved that locally flat~$\Z$-surfaces of the same genus are unique if~$b_2(M) \geq |\sigma(M)|+6$ and that locally flat~$\Z_n$-surfaces of the same genus and homology class are unique if~$b_2(M) \geq |\sigma(M)|+~2$ and the surfaces are not of minimal genus (among locally flat~$\Z_n$-surfaces in the given homology class)~\cite[Theorems 7.2 and 7.4]{Sunukjian}.
In fact, locally flat~$\Z$-surfaces of the same genus are determined up to ambient isotopy by their equivariant intersection form~\cite[Theorem~1.4]{ConwayPowell}.
None of these results covers~$\Z_2$-spheres in $\C P^2$.

\subsection{Comparison with the smooth category}

The class~$d \in H_2(\C P^2)$ is represented by a smoothly embedded sphere if and only if~$d \in \lbrace 0,1,2 \rbrace$~\cite[page 264]{Tristram}.
Kronheimer-Mrowka showed that the minimum genus of a smoothly embedded surface in~$\C P^2$ representing the homology class~$d$ is~$(d-1)(d-2)/2$~\cite[Theorem 1]{KronheimerMrowka}. For~$d>2$,  Kim proved that minimal genus representatives admit exotic copies~\cite[Corollary 3.5 and Theorem 4.5]{KimTwistSpinning} (earlier, for $d\geq 5$, Finashin constructed smoothly distinct minimal genus surfaces, but did not prove that they were topologically isotopic~\cite[Theorem 1.1 and Remark on page 50]{Finashin}).
The case of spheres remained open until very recently when Miyazawa~\cite[Theorems~1.5 and~4.40]{Miyazawa} produced smoothly distinct $\Z_2$-spheres in $\C P^2$. 
Miyazawa's spheres are all topologically isotopic:
they are constructed as connected sums of a holomorphic~$\Z_2$-sphere in~$\C P^2$ with a~$0$-twist $1$-roll spun $2$-knot in $S^4$,  the outcome of this connected sum is a cyclic rim surgery and so one can apply~\cite[Theorem~1.3]{KimRubermanTopologicalTriviality} to conclude topological uniqueness.\footnote{We are grateful to Mark Powell for pointing this out this application of~\cite{KimRubermanTopologicalTriviality}.}
Theorem~\ref{thm:Main} also applies to Miyazawa's examples, providing an alternative proof that they are topologically isotopic. 
We hope the more general nature of Theorem~\ref{thm:Main} will prove the topological triviality of more general examples, e.g.~$\Z_2$-spheres that do not arise as such connected sums.

We also note that the 1997 edition of Kirby's problem list asks whether every minimal genus representative of~$d \in H_2(\C P^2)\cong \Z$ is smoothly isotopic to an algebraic curve and remarks that the topological case is open~\cite[Problem 4.110 ]{KirbyProblemList}.
In the smooth category,  the answer is negative for certain~$d \geq 3$~\cite{Finashin} and for~$d=2$~\cite{Miyazawa},  whereas in the topological category, the answer is negative for every even~$d>4$: for such $d$, Lee-Wilczy\'{n}ski showed that the minimal genus, which in this case equals~$\lfloor d^2/4\rfloor-1$, is realised; in particular they produce surfaces of positive genus that are not topologically isotopic to an algebraic curve~\cite[Corollary~1.3]{LWGenus}.
On the other hand,  the combination of our Theorem~\ref{thm:Main} and~\cite[Theorem~1.2]{LeeWilczy} shows that for spheres, i.e.~for~$d=\pm 1,\pm 2$ (there is no degree $0$ algebraic curve), the answer to Kirby's Problem~4.110 is positive in the topological category.
The question appears to be still open in the topological category for~$d\geq 3$ odd and for~$d=4$.

\subsection{Proof strategy}
The proof of Theorem~\ref{thm:Main} relies on Kreck's modified surgery~\cite{KreckSurgeryAndDuality}.
We recall the main steps,  assuming some familiarity with this theory and note that our strategy is similar to the one used in~\cite{ConwayOrsonPowell} to unknot nonorientable surfaces in the $4$-sphere.
The exterior of a~$\Z_2$-sphere in~$\C P^2$ is almost spin (Proposition~\ref{prop:AlmostSpin}), so its normal~$1$-type is a certain fibration~$p \colon B \to \BSTOP$,  whose definition is recalled in Section~\ref{sub:Normal1Type}.
Given~$\Z_2$-spheres~$F_0,F_1 \subseteq \C P^2$ with exteriors~$X_0,X_1$,  we show that every orientation-preserving homeomorphism~$f \colon \partial X_0 \to \partial X_1$ extends over the tubular neighbourhoods and that~$X_0 \cup_f -X_1$ is again almost spin with fundamental group $\Z_2$ (Proposition~\ref{prop:UnionSpin}).
We combine this fact with an analysis of the bordism group~$\Omega_4(B,p)$ to show that there exists a $(B,p)$-bordism~$(W,\overline{\nu})$ between normal $1$-smoothings for $X_0$ and~$X_1$. 
The associated modified surgery obstruction to modifying~$W$ to be an $s$-cobordism lies in the~$\ell$-monoid~$\ell_5(\Z[\Z_2])$. 
We show in Proposition~\ref{prop:SurfaceData} that the obstruction in fact lies in a certain subset of the monoid, denoted~$\ell_5(\Z[\Z_2],(T-1))$ in the formalism of~\cite{CrowleySixt}.
We then apply a criterion from~\cite{ConwayOrsonPowell} which implies that if~$\ell_5(\Z,-2) \subseteq \ell_5(\Z)$ is trivial, then every element of~$\ell_5(\Z[\Z_2],(T-1))$ is elementary.
Using that the monoid~$\ell_5(\Z)$ is closely related to quadratic linking forms~\cite[Section~6]{CrowleySixt}, we are then able to show that~$\ell_5(\Z,-2)$ is trivial and apply said criterion.
Thus $(W,\overline{\nu})$ is $(B,p)$-bordant rel.~boundary to an $s$-cobordism and, since~$\Z_2$ is a good group,  the~$5$-dimensional $s$-cobordism theorem~\cite{FreedmanQuinn} implies that the homeomorphism~$f$ extends to an orientation-preserving homeomorphism~$\Phi \colon X_0 \to X_1$.
As~$f$ extends over the tubular neighborhoods of the surfaces, we
deduce that $\Phi(F_0)=F_1$.
We then argue that this self-homeomorphism of $\C P^2$ can be assumed to induce the identity on $H_2(\C P^2)$.
Work of Kreck, Perron and Quinn~\cite{MR561244,Perron,Quinn} on the mapping class groups of simply connected, closed $4$-manifolds then applies, to show that $\Phi$ is isotopic to the identity, and so the surfaces are ambiently isotopic.

We remark that it is likely a proof of Theorem~\ref{thm:Main} could also be obtained using the normal~$2$-type instead of the normal $1$-type, but we will not pursue this further here.
We adopted the normal~$1$-type approach here because we found it to be efficient in this case.

\subsection*{Organisation}

In Section~\ref{sec:AlgTop},  we determine the algebraic topology of $\Z_2$-sphere exteriors.
In Section~\ref{sec:Union},  we prove that such exteriors are almost spin as is the union of two such exteriors.
In Section~\ref{sec:ModifiedSurgery},  we apply the modified surgery programme to prove Theorem~\ref{thm:Main}.
Finally Appendix~\ref{sec:Zspheres} proves that there is a unique $\Z$-sphere in $\C P^2$.

\subsection*{Acknowledgements}
AC was partially supported by the NSF grant DMS\unaryminus 2303674.
We are grateful to Marco Golla for a helpful conversation concerning algebraic curves.

\subsection*{Conventions}
The notation $\Z_n$ denotes the cyclic group of order $n$. 
We work in the topological category throughout. 
The term \emph{locally flat} implicitly always includes the idea of being the image of a topological embedding (rather than allowing the possibility of locally flat immersion). 
Manifolds are assumed to be compact, connected and oriented.

\section{The algebraic topology of the exterior}
\label{sec:AlgTop}

Let~$F \subseteq \C P^2$ be a~$\Z_2$-sphere. Write~$\nu(F)$ for an open tubular neighbourhood.
We describe the algebraic topology of the
exterior~$X:=\C P^2 \setminus  \nu (F)$ and its universal cover~$\widetilde{X}$ as well as the corresponding intersection forms.
We have fixed an isomorphism $H_2(\C P^2)\cong\Z$. 
From now on we assume that $[F]=2$, the proofs for $[F]=-2$ being entirely analogous.

\medbreak

In what follows, $Q_W$ denotes the intersection form of a $4$-manifold $W$.
As~$[F]=2\in H_2(\C P^2)$, the self-intersection is~$Q_{\C P^2}([F],[F])=4$.
Thus the closed tubular neighbourhood~$\overline{\nu}(F)$ may be identified with~$D^2\widetilde{\times}_4 S^2$, the~$D^2$-bundle over~$S^2$ with Euler number~$4$. 
Choose such an identification. 
Using this choice we have an identification of $\partial X$ with the lens space~$L(4,1)$,
and we now also fix an identification $\pi_1(L(4,1))=\Z_4$.
As both~$\pi_1(\partial X)$ and~$\pi_1(X)$ are generated by the class of a meridian to~$F$, it follows that the inclusion induced map on~$\pi_1$ is the surjection~$\Z_4\to \Z_2$.

\begin{proposition}
\label{prop:IntegralHomologyGroup}
Given a~$\Z_2$-sphere~$F \subseteq \C P^2$, with exterior~$X$,  the following assertions hold:
\begin{enumerate}
\item we have~$\pi_1(X)=\Z_2=H_1(X);$
\item we have~$H_2(X)=0$ and~$H_3(X)=0$.
\end{enumerate}
\end{proposition}
\begin{proof}
The first assertion is immediate. 
As $H_1(X)$ is finite, the universal coefficient theorem implies $H^1(X)=0$. As $X$ has connected boundary, the long exact sequence of the pair 
now gives $H^1(X,\partial X)=0$. So then $H_3(X)=0$ by Poincar\'{e} duality.
We prove that~$H_2(X)=0$.
As~$H_2(\partial X)=0$, the Mayer-Vietoris exact sequence for~$\C P^2=X \cup \overline{\nu}(F)$ reduces to
\[
0 \to H_2(X) \oplus H_2(\overline{\nu}(F)) \to H_2(\C P^2).
\]
As~$H_2(\C P^2)\cong\Z$ and the second map is an injection, we deduce that~$H_2(X)$ is free. As~$H_2(\overline{\nu}(F))\cong\Z$ and the second map is an injection,~$H_2(X)=0$, for rank reasons.
\end{proof}

In what follows, we write~$\Z_2=\langle T \mid T^2=1 \rangle$ and use~$\Z_-$ to denote~$\Z$ with the~$\Z[\Z_2]$-module structure induced by~$T \cdot n=-n$, for every~$n \in \Z$.

\begin{proposition}
\label{prop:H2UnivCover}
Given a~$\Z_2$-sphere~$F \subseteq \C P^2$, with exterior~$X$, we have
\[
\pi_2(X) \cong H_2(\widetilde{X})\cong\Z_-.
\]
\end{proposition}
\begin{proof}
The first isomorphism is a consequence of the Hurewicz theorem and so we focus on proving that $H_2(\widetilde{X})\cong\Z_-.$
We first calculate~$H_2(\widetilde{X})$ as an abelian group.
Since~$\widetilde{X}$ is simply-connected, we see that~$H_3(\widetilde{X})=0$ and~$H_2(\widetilde{X})$ is free.
To deduce that~$H_2(\widetilde{X})$ is free of rank~$1$,  we use Proposition~\ref{prop:IntegralHomologyGroup} to obtain
\[
2=2\chi(X)=\chi(\widetilde{X})=1+b_2(\widetilde{X}).
\]
We now determine the~$\Z[\Z_2]$-module structure of~$H_2(\widetilde{X})$.

We claim that~$H_2(\widetilde{X};\Q)=\Q_-$.
We use~$\mathcal{E}_\pm(H)$ to denote the~$(\pm 1)$-eigenspaces of an endomorphism~$T \colon H \to H$ with~$T^2=1$.
Using~\cite[Chapter 3, Theorem~7.2]{BredonIntroduction} together with Proposition~\ref{prop:IntegralHomologyGroup}, we obtain
$$\mathcal{E}_+(H_2(\widetilde{X};\Q))=H_2(\widetilde{X};\Q)^{\Z_2}\cong H_2(X;\Q)=0.$$
Since we have already established that~$\operatorname{dim}_\Q H_2(\widetilde{X};\Q)=1$, it follows that
$$\Q
=H_2(\widetilde{X};\Q)
=\mathcal{E}_+(H_2(\widetilde{X};\Q))  \oplus \mathcal{E}_-(H_2(\widetilde{X};\Q)) 
=\mathcal{E}_-(H_2(\widetilde{X};\Q)).
$$
This establishes our claim that~$H_2(\widetilde{X};\Q)=\Q_-$.

Using the fact that~$H_2(X) \cong \Z$ is torsion free as an abelian group, a theorem of Reiner~\cite{Reiner} implies that as a~$\Z[\Z_2]$-module,~$H_2(X)$ decomposes as~$P \oplus P_+ \oplus P_-$. 
Here,~$P$ is~$\Z[\Z_2]$-projective,~$P_+$ has the trivial~$\Z[\Z_2]$-action and~$P_-$ has the~$\Z[\Z_2]$-action where~$T$ operates by~$-1$.
The claim implies that~$P=0$ and~$P_+=0$. 
We obtain~$H_2(\widetilde{X})=\Z_-$, as required.
\end{proof}

\begin{proposition}
\label{prop:IntersectionForm}
Given a~$\Z_2$-sphere~$F \subseteq \C P^2$, with exterior~$X$, the universal cover has intersection form~$Q_{\widetilde{X}}=(- 2)$ and thus~$\widetilde{X}$ is spin.
\end{proposition}
\begin{proof}
Fix a generator for~$H_2(\widetilde{X})\cong\Z$, the dual generator for~$H_2(\widetilde{X})^*$, and the Poincar\'e dual generator for~$H_2(\widetilde{X},\partial \widetilde{X})$.
With respect to these generators, the matrix for~$Q_{\widetilde{X}}$ is the same as the~$1\times1$ matrix for the inclusion-induced map~$H_2(\widetilde{X}) \to H_2(\widetilde{X},\partial \widetilde{X})$.
We calculate the latter.

Since~$\partial \widetilde{X}$ is the cover of~$\partial X=L(4,1)$ corresponding to the kernel of the epimorphism~$\Z_4 \twoheadrightarrow \Z_2$, it is homeomorphic to~$L(2,1) \cong \R P^3$. In particular,~$H_1(\partial \widetilde{X})=\Z_2$ and~$H_2(\partial \widetilde{X})=0$.
The long exact sequence of the pair~$(\widetilde{X},\partial \widetilde{X})$ thus yields the short exact sequence
$$ 0 \to H_2(\widetilde{X}) \to H_2(\widetilde{X},\partial \widetilde{X}) \to \Z_2 \to 0.$$
We deduce that any matrix for the inclusion induced map~$H_2(\widetilde{X}) \to H_2(\widetilde{X},\partial \widetilde{X})$ is of the form~$(\pm 2)$. 
As we explained above, this implies that~$Q_{\widetilde{X}}=(\pm 2).$

It remains to determine whether~$Q_{\widetilde{X}}$ is represented by~$(2)$ or~$(-2)$.
Write $\Sigma_2(F)$ for the double cover of $\C P^2$, branched over $F$, and~$N_2(F)\subseteq \Sigma_2(F)$ for the preimage of~$\ol{\nu}(F)\subseteq \C P^2$ under the branched double covering map. By Novikov addivity we have
\begin{equation}
\label{eq:Signature}
\sigma(\Sigma_2(F))=\sigma(\widetilde{X})+\sigma(N_2(F)).
\end{equation}
We claim that as~$F\subseteq \C P^2$ has Euler number $4$, so~$F\subseteq \Sigma_2(F)$ has Euler number $2$. The analogous multiplicativity fact was shown in~\cite[Proposition~4.5, Claim 1]{ConwayOrsonPowell}, when~$F$ is a nonorientable surface. Setting~$F$ to be a~$2$-sphere does not affect the proof in~\cite{ConwayOrsonPowell}, which can thus be used verbatim in our case. Hence~$F\subseteq \Sigma_2(F)$ has self-intersection $2$, as claimed.

The claim implies that $N_2(F)$ is a $D^2$-bundle over $S^2$ with Euler number $2$ and therefore~$\sigma(N_2(F))=1$.
On the other hand, using~\cite[Theorem 1]{GeskeKjuchukovaShaneson}, we compute
\[
\sigma(\Sigma_2(F))= 2\sigma(\C P^2)-\frac{1}{2}e(F)=2-2=0.
\]
It follows from~\eqref{eq:Signature} that~$\sigma(\widetilde{X})=-1$ and thus~$Q_{\widetilde{X}}=(-2)$, as claimed.
The fact that~$\widetilde{X}$ is spin follows because it is a simply-connected~$4$-manifold with even intersection form.
\end{proof}

\begin{corollary}
\label{cor:EquivariantIntersectionForm}
The equivariant intersection form of the exterior of a~$\Z_2$-sphere~$F \subseteq \C P^2$ is isometric to the Hermitian pairing
\begin{align*}
\Z_- \times \Z_- &\to \Z[\Z_2]  \\
(x,y) & \mapsto -2(1-T)xy.
\end{align*}
\end{corollary}
\begin{proof}
This follows immediately by combining  the definition of the equivariant intersection form with Propositions~\ref{prop:H2UnivCover} and~\ref{prop:IntersectionForm}: for every~$x,y \in H_2(\widetilde{X}) \cong \Z_-$, we have
$$\lambda_{X}(x,y)=Q_{\widetilde{X}}(x,y)+TQ_{\widetilde{X}}(x,Ty)=(1-T)Q_{\widetilde{X}}(x,y)=-2(1-T)xy.~$$
\end{proof}


\section{An almost spin union of sphere exteriors}
\label{sec:Union}

A manifold is called \emph{almost spin} if its universal cover is spin, but the manifold itself is not spin. In this section we show that
exteriors of~$\Z_2$-spheres are almost spin and that glueing two such exteriors together also results in an almost spin manifold.

\begin{proposition}
\label{prop:AlmostSpin}
The exterior~$X$ of a~$\Z_2$-sphere~$F \subseteq \C P^2$ is almost spin.
\end{proposition}
\begin{proof}
Since~$Q_{\C P^2}(1,[F])=2$ and~$Q_{\C P^2}(1,1)=1$ do not agree modulo 2, we see that~$[F]$ is not a characteristic class. 
It then follows that~$w_2(X)\neq 0$; see e.g.~\cite[Lemma 3.4]{OrsonPowellSpines}.
Hence~$X$ is not spin. On the other hand, we showed in Proposition~\ref{prop:IntersectionForm} that~$\widetilde{X}$ is spin.
\end{proof}

The next proposition concerns unions of $\Z_2$-sphere exteriors along their boundary.

\begin{proposition}\label{prop:UnionSpin}
For~$i=0,1$, let~$F_i\subseteq \C P^2$ be a~$\Z_2$-sphere with $[F_i]=2\in H_2(\C P^2)$, and with exterior~$X_i$.
Let~$f\colon \partial X_{0} \to \partial X_{1}$ be any orientation-preserving homeomorphism, and define~$M:=X_{0} \cup_f -X_{1}$. Then the following hold:
\begin{enumerate}
\item\label{item:1spin} 
the map~$f$ extends to an orientation-preserving homeomorphism~$\overline{\nu}(F_0)\to \overline{\nu}(F_1)$, which restricts to a homeomorphism $F_0\to F_1$;
\item\label{item:2spin} the inclusion induced maps~$\pi_1(X_i)\to \pi_1(M)$ are isomorphisms (in particular~$\pi_1(M)\cong\Z_2$, generated by a meridian of~$F_i$); and
\item\label{item:3spin}
the union~$M$ is almost spin.
\end{enumerate}
\end{proposition}
\begin{proof}
For~$i=0,1$, fix an identification between~$\overline{\nu}(F_i)$ and~$D^2\widetilde{\times}_4 S^2$, the~$2$-disc bundle over the~$2$-sphere with Euler number 4. Choose once and for all an identification~$\pi_1(L(4,1))\cong \Z_4$. By \cite[Th\'{e}or\`{e}me 3(c)]{Bonahon}, the
homeomorphism $f$ is isotopic to the identity map or to a certain smooth involution, which we will call~$\tau$. Both of these maps extend to orientation-preserving bundle self-isomorphisms of~$D^2\widetilde{\times}_4 S^2$ (see e.g.~\cite[Definition~4.6]{OrsonPowellSpines}). This shows that \eqref{item:1spin} holds, up to an isotopy. But we may insert this isotopy in a boundary collar of~$\overline{\nu}(F_0)$, so that indeed \eqref{item:1spin} holds, on the nose.

On fundamental groups, we have~$f_*=\pm1\colon\Z_4\to\Z_4$, with $f_*=1$ if it is isotopic to the identity and~$f_*=-1$ if it is isotopic to $\tau$ (see e.g.~\cite[Lemma~4.7]{OrsonPowellSpines}). As~$\pm 1\equiv 1\pmod 2$, we see that~$f_*$ intertwines the inclusion induced maps $\pi_1(\partial X_i)\to\pi_1(X_i)=\Z_2$. This firstly shows, via a van-Kampen argument, that \eqref{item:2spin} holds. It secondly shows that~$f$ lifts to an orientation-preserving homeomorphism~$\widetilde{f}\colon \partial \widetilde{X}_0\to \partial \widetilde{X}_1$, such that~$\widetilde{M}=\widetilde{X}_{0} \cup_{\widetilde{f}}-\widetilde{X}_{1}$. 

Since the~$X_{i}$ are not spin, neither is the union~$M$. We claim that~$\widetilde{M}$ is spin. 
The problem of whether the union of simply-connected~$4$-manifolds along a rational homology sphere boundary is spin has been addressed by Boyer~\cite{BoyerUniqueness}. Note that the non-trivial double cover of $L(4,1)$ is $L(2,1)$,  a rational homology sphere.
Any isometry $\varphi \colon Q_{\widetilde{X}_{0}} \cong Q_{\widetilde{X}_{1}}$ induces a boundary isometry $\partial \varphi \colon H_1(\partial \widetilde{X}_0) \to H_1(\partial \widetilde{X}_1)$.
Since $H_1(\partial \widetilde{X}_i)\cong \Z_2$, there is a unique such isomorphism and therefore
$$\partial \varphi=\widetilde{f}_* \colon H_1(\partial \widetilde{X}_0) \to H_1(\partial \widetilde{X}_1).$$
It follows that, in Boyer's language, the pair~$(f,\varphi)$ is a ``morphism"~\cite[p.~334]{BoyerUniqueness}.
Since~$\widetilde{X}_{0}$ is simply-connected with even intersection form and the~$L(2,1)$ is a rational homology sphere, we deduce from~\cite[Proposition 0.8, combination of (i)+(iii)]{BoyerUniqueness} that~$\widetilde{M}$ is spin.
This concludes the proof of the claim and of item~\eqref{item:3spin}. 
\end{proof}

\section{Modified Surgery}
\label{sec:ModifiedSurgery}

We assume some familiarity with modified surgery theory~\cite{KreckSurgeryAndDuality}.
In particular,  we assume that the reader is familiar with normal smoothings and normal~$k$-types~\cite[Section 2]{KreckSurgeryAndDuality}; see also~\cite{KasprowskiLandPowellTeichner} for a nice exposition in the context of $4$-manifolds.
We apply the modified surgery programme to prove that homologous $\Z_2$-spheres are isotopic: in Section~\ref{sub:Normal1Type} we prove that $\Z_2$-sphere exteriors are bordant over their normal $1$-type,  whereas in Section~\ref{sub:Monoid} we conclude by analysing the $\ell$-monoid.

\subsection{The normal~$1$-type}
\label{sub:Normal1Type}

Let~$M$ be an almost spin~$4$-manifold with fundamental group~$\Z_2$. Choose a map~$c\colon M\to \R P^\infty$ inducing the isomorphism on fundamental groups. By \cite[Lemma 3.17]{KasprowskiLandPowellTeichner}, there is a unique class~$w\in H^2(\R P^\infty;\Z_2)$ such that~$c^*(w)=w_2(M)$. In an abuse of notation, also write~$w\colon \R P^\infty\to K(\Z_2,2)$ for a map inducing the cohomology class. Consider the following diagram
\[
\begin{tikzcd}[column sep = scriptsize]
\BTOPSpin\ar[r,"i"] \ar[d,"="]&B \arrow[r, "q"] \arrow[d, "p"]
\arrow[dr, phantom, "\scalebox{1.5}{$\lrcorner$}" , very near start, color=black]
& \R P^\infty \arrow[d, "w"] \\
\BTOPSpin\ar[r, "p\circ i"] &\BSTOP \arrow[r, "w_2"] & K(\Z_2,2).
\end{tikzcd}
\]
The CW complex~$B$, and maps~$p,q$, are chosen to make the right-hand square a homotopy pullback. 
The map~$w_2$ is chosen to be a fibration, with fibre~$\BTOPSpin$. 
This implies the map~$q$ is also a fibration. 
We choose the space~$B$ so that the map~$i$ is the inclusion of a fibre of~$q$, and so that the left-hand downward map is indeed an equality. 
We note, for later use, that $\pi_2(B)=0$; this may be readily computed from the long exact sequence of the fibration~$q$ and the fact that $\BTOPSpin$ is $2$-connected.
In particular, this means that a map $f\colon M\to B$ is $2$-connected if and only if it is an isomorphism on $\pi_1$.

\medbreak

The following justifies that~$(B,p)$ is the correct normal~$1$-type for us.

\begin{proposition}[{\cite[Theorem~2.2.1(b) (I)]{TeichnerThesis}}]
\label{prop:Normal1Type}
Let~$M$ be an almost spin~$4$-manifold with fundamental group~$\Z_2$. 
Then the fibration~$(B,p)$ is a model for the normal~$1$-type of~$M$.
\end{proposition}

Using this normal~$1$-type we can begin the modified surgery programme for producing homeomorphisms between~$\Z_2$-sphere exteriors. We first obtain a normal~$(B,p)$-bordism between normal~$1$-smoothings.

\begin{proposition}
\label{prop:StablyHomeo}
Let~$F_0$ and~$F_1$ be~$\Z_2$-spheres with $[F_i]=2\in H_2(\C P^2)$, and let~$f \colon \partial X_{0} \to \partial X_{1}$ be any orientation-preserving homeomorphism. Then there exist normal~$1$-smoothings~$(X_{i}, \ol{\nu}_i)$ for~$i=0,1$, that are bordant over their normal~$1$-type, relative to~$f$.
\end{proposition}

\begin{proof}
Write~$(B,p)$ for the normal~$1$-type described in Proposition~\ref{prop:Normal1Type}. By Proposition~\ref{prop:UnionSpin}, $M:=X_{0} \cup_{f} -X_{1}$ is almost spin and has fundamental group~$\Z_2$. Choose a normal~$1$-smoothing~$\ol{\nu}\colon M\to B$, in other words a~$2$-connected normal~$B$-structure.
For~$i=0,1$, denote by~$\ol{\nu_i}\colon X_{i}\to B$ the restriction of~$\ol{\nu}$, and note this is a normal~$B$-structure. 
We argue that the~$\ol{\nu_i}$ are moreover normal $1$-smoothings. As noted above, because $\pi_2(B)=0$,  the $2$-connectivity of $\ol{\nu_i}$ is equivalent to the statement that~$\ol{\nu_i}$ is an isomorphism on~$\pi_1$. 
But~$\ol{\nu}$ is an isomorphism on~$\pi_1$ and, by Proposition~\ref{prop:UnionSpin}, so are the inclusions~$X_i\to M$. Thus we may conclude that~$(X_i,\ol{\nu}_i)$ is a normal~$1$-smoothing for~$i=0,1$, as claimed.

Using these normal~$1$-smoothings, the proposition now reduces to proving that~$(M,\overline{\nu})$ vanishes in the bordism group~$\Omega_4(B,p)$. 

A spectral sequence calculation~\cite[Theorem~4.4.4]{TeichnerThesis}, combined with~\cite[Proposition~4]{TeichnerOnTheSignature}, shows that
\[
\Omega_4(B,p)\cong \left( 8\Z\oplus H_4(\Z_2;\Z)\right) \oplus \Z_2\cong 8\Z\oplus\Z_2,
\]
where the final isomorphism uses that $H_4(\Z_2;\Z)=0$. Under these isomorphisms, the bordism class~$[(M,\overline{\nu})]$ is given in the first summand by the signature~$\sigma(M)$ and in the second by the Kirby-Siebenmann invariant $\ks(M)$. 
As~$H_2(X_i)=0$, in particular~$\sigma(X_{i}) =0$, so by Novikov additivity,~$\sigma(M)=0$.
By additivity of~$\ks$ (see e.g.~\cite[Theorem 8.2]{FriedlNagelOrsonPowell}), we deduce from~$\C P^2=X_i\cup\overline{\nu}(F_i)$ that~$\ks(X_i)=0$. Applying additivity again, implies that~$\ks(M)=0$.
It follows that~$(M,\overline{\nu})$ is~$(B,p)$-null-bordant, as claimed.
\end{proof}

\subsection{The analysis of the monoid.}
\label{sub:Monoid}

We recall some facts from Kreck's modified surgery theory~\cite{KreckSurgeryAndDuality,CrowleySixt}.
We say that~$(W,\overline{\nu},M_0,\overline{\nu}_0,M_1,\overline{\nu}_1)$ is a \emph{modified surgery problem} if~$M_0$ and~$M_1$ are~$4$-manifolds with normal~$1$-type~$(B,p)$,  the~$\overline{\nu}_i \colon M_i \to B$ are normal~$1$-smoothings, and~$(W,\overline{\nu})$ is a~$(B,p)$-cobordism between~$(M_0,\overline{\nu}_0)$ and ~$(M_1,\overline{\nu}_1)$.
Kreck defines an obstruction to~$(W,\overline{\nu})$ being~$(B,p)$-bordant rel.~boundary to an~$s$-cobordism~\cite[Theorem 4]{KreckSurgeryAndDuality}.
We describe this obstruction very briefly referring to~\cite[Sections 6 and 7]{KreckSurgeryAndDuality} for more details as well as to~\cite[Sections 3, 7 and 8]{ConwayOrsonPowell} for an exposition resembling this one.
We assume some familiarity with quadratic forms and their lagrangians.
In what follows,  all modules are assumed to be stably free.

\begin{itemize}[leftmargin=*]\setlength\itemsep{0em}
\item For every group~$\pi$ there is a monoid~$\ell_5(\Z[\pi])$ whose elements are equivalence classes of triples~$((M,\psi);F,V)$ where~$(M,\psi)$ is a quadratic form over~$\Z[\pi]$,~$F \subseteq  M$ is a lagrangian and~$V \subseteq  M$ is a half-rank direct summand.
These triples are called \emph{quasiformations}.
We refer to~\cite{CrowleySixt, ConwayOrsonPowell} for further details (in particular for questions about Whitehead torsion) but note that an element~$x \in \ell_5(\Z[\pi])$ is called \emph{elementary} if it is represented by a quasiformation~$((M,\psi);F,V)$ where the inclusions induce an isomorphism~$F \oplus V \cong M$.
Note that in our case $\pi=\Z_2$ has vanishing Whitehead group which is why we do not concern ourselves with Whitehead torsion.
\item Given a modified surgery problem~$(W,\overline{\nu},M_0,\overline{\nu}_0,M_1,\overline{\nu}_1)$,  Kreck~\cite[Theorem 4]{KreckSurgeryAndDuality} defines a class~$\Theta(W,\overline{\nu}) \in \ell_5(\Z[\pi_1(B)])$, called the \emph{modified surgery obstruction}, that only depends on the~$(B,p)$-bordism rel. boundary class of~$(W,\overline{\nu})$ and that is elementary if and only if~$(W,\overline{\nu})$ is~$(B,p)$-bordant rel.~boundary to an~$s$-cobordism. 
\item We recall more facts about the monoid in order to describe a subset of~$\ell_5(\Z[\pi_1(B)])$ that contains the modified surgery obstruction.
Two quadratic forms~$(P,\psi)$ and~$(P',\psi')$ are \emph{$0$-stably equivalent} if there are zero forms~$(Q,0)$ and~$(Q',0)$ and an isometry 
$$(P,\psi) \oplus (Q,0) \cong (P',\psi') \oplus (Q',0).$$
For every quadratic form~$v'=(V',\theta')$ over~$\Z[\pi]$,  we write~$\ell_5(v') \subseteq  \ell_5(\Z[\pi])$ for the subset of equivalence classes of quasiformations $((M,\psi);F,V)$ such that the \emph{induced forms}~$(V,\psi|_{V \times V})$ and~$(V^\perp,\psi|_{V^\perp \times V^\perp})$ are~$0$-stably equivalent to~$(V',\theta')$.
For every~$0$-stable equivalence class of a quadratic form~$v'=(V',\theta')$,  the subset~$\ell_5(v')\subset \ell_5(\Z[\pi])$ contains an elementary class~\cite[Corollary 5.3]{CrowleySixt}; in particular it is nonempty.
\item Next we recall how this algebra relates to the modified surgery obstruction.
Set $\pi:=\pi_1(B)$ and write 
$$K\pi_2(M_i):=\ker((\overline{\nu}_i)_* \colon \pi_2(M_i) \to \pi_2(B)).$$
Intersections and self-intersections in~$M_i$ define a quadratic form~$(K\pi_2(M_i),\psi_{M_i})$ called the \emph{Wall form} of~$(M_i, \overline{\nu}_i)$~\cite[Section 5]{KreckSurgeryAndDuality}.
Write~$\lambda_{M_i}$ for the symmetrisation of~$\psi_{M_i}$ (this agrees with the restriction of the equivariant intersection form on~$\pi_2(M_i) \cong H_2(M_i;\Z[\pi])$).
For simplicity, assume that~$\lambda_{M_i}$ is nondegenerate.
Given a free~$\Z[\pi]$-module~$S_i$ and a surjection~$\varpi_i \colon S_i \twoheadrightarrow K \pi_2 (M_i)$, the pull-back of the (nondegenerate) Wall form~$(K\pi_2(M_i),\psi_{M_i})$ by~$\varpi_i$ is called a {\em free Wall form} of~$(M_i, \overline{\nu}_i)$.

As explained in~\cite{CrowleySixt} (see also~\cite[Section 8]{ConwayOrsonPowell}),  another result of Kreck~\cite[Proposition 8]{KreckSurgeryAndDuality} implies that if~$(W,\overline{\nu},M_0,\overline{\nu}_0,M_1,\overline{\nu}_1)$ is a modified surgery problem with modified surgery obstruction~$\Theta(W,\overline{\nu})\in \ell_5(\Z[\pi])$,  then for any choice of free Wall forms~$v(\overline{\nu}_0)$ and~$v(\overline{\nu}_1)$ for~$(M_0,\overline{\nu}_0)$ and~$(M_1,\overline{\nu}_1)$ 
that are~$0$-stably isometric to a form~$v'$,
we have
\[\Theta(W,\overline{\nu}) \in \ell_{5}(v').\]
\end{itemize}
\begin{remark}
Given $4$-manifolds $M_0$ and $M_1$ that are bordant over their normal $1$-type and have isometric equivariant intersection,
the take-away is that if~$(M_0,\ol{\nu}_0)$ and~$(M_1,\ol{\nu}_1)$ have nondegenerate Wall forms (and therefore isometric free Wall forms), then a strategy for showing that~$M_0$ and~$M_1$ are $s$-cobordant is to find a quadratic form $v'$ isometric to the free Wall forms of~$M_0$ and~$M_1$ and to show that the subset~$\ell_{5}(v') \subset \ell_5(\Z[\pi])$ is trivial.
\end{remark}

Let~$F \subseteq  \C P^2$ be a~$\Z_2$-sphere.
We describe a specific free Wall form for the exterior~$X$.
We calculated in Proposition~\ref{prop:IntersectionForm} that~$Q_{\widetilde{X}}(x,y)=-2xy$; in particular~$Q_{\widetilde{X}}$ is even and nondegenerate. The unique nondegenerate quadratic form over~$\Z$ refining~$Q_{\widetilde{X}}$ is~$(\Z,\theta_{\widetilde{X}})$, where~$\theta_{\widetilde{X}}(x,y)=-xy$.
By Proposition~\ref{prop:H2UnivCover} we have~$H_2(\widetilde{X}) \cong\Z_-$ and, remembering the~$\Z[\Z_2]$-module structure in order to identify~$\Z_-=\Z[\Z_2] \otimes_{\Z[\Z_2]} \Z_-$, we will think of the integral quadratic form~$\theta_{\widetilde{X}}$ as a pairing:
\begin{equation}\label{eqn:theta-nd-rho}
\theta_{\widetilde{X}} \colon (\Z[\Z_2] \otimes_{\Z[\Z_2]} \Z_-) \times (\Z[\Z_2] \otimes_{\Z[\Z_2]} \Z_-) \to \Z.
\end{equation}
Occasionally we will write this form as~$\theta_{\widetilde{X}}=(-1)$.
Using our calculation of the equivariant intersection form from Corollary~\ref{cor:EquivariantIntersectionForm}, the next proposition establishes a formula for a free Wall form for~$X$.

\begin{proposition}
\label{prop:SurfaceData}
Let~$F \subseteq \C P^2$ be a~$\Z_2$-sphere and let~$\overline{\nu} \colon X \to B$ be a normal~$1$-smoothing.
Then a free Wall form for~$X$ is given by~$(\Z[\Z_2],\vartheta)$, where the quadratic form~$\vartheta$ is the pairing
\[
 \Z[\Z_2] \times \Z[\Z_2]\to \Z[\Z_2],
\qquad (x,y)\mapsto
(1-T)\theta_{\widetilde{X}}(x \otimes 1,y\otimes 1).
\]
We write~$\vartheta=(1-T)\theta_{\widetilde{X}}=(1-T)(-1).$
\end{proposition}
\begin{proof}
The proof is entirely analogous as~\cite[Proposition 10.2]{ConwayOrsonPowell}.
Note that~\cite[Proposition 10.2]{ConwayOrsonPowell} is concerned with spin manifolds, but the proof is identical for almost spin manifolds, recalling, as noted earlier, in our case it is also true that~$\pi_2(B)=0$ and that we've calculated the equivariant intersection form $\lambda_X$ in Corollary~\ref{cor:EquivariantIntersectionForm}.
Also,  in~\cite[Proposition~10.2]{ConwayOrsonPowell} the Wall form was degenerate so it was necessary to mod out by the radical; in our setting the Wall form is nondegenerate and so~\cite[Proposition~10.2]{ConwayOrsonPowell} sould be read with~$\theta=\theta^{\nd}$.
\end{proof}

Combining our recollections of modified surgery with Proposition~\ref{prop:SurfaceData} implies that, given a modified surgery problem~$(W,\overline{\nu},X_0,\overline{\nu}_0,X_1,\overline{\nu}_1)$, the modified surgery obstruction~$\theta(W,\overline{\nu})$ belongs to~$\ell_5(\Z[\Z_2],(1-T)\theta)$ with~$\theta=(-1)$.
In order to show that all the elements of this subset are elementary,  we use the following criterion.

\begin{proposition}[{\cite[Proposition 9.15]{ConwayOrsonPowell}}]
\label{prop:NewElementaryCriterion}
If~$(\Z^h,\theta)$ is a nondegenerate quadratic form over~$\Z$ such that~$\ell_5(\Z^h,2\theta)$ is a singleton, then every~$\Theta \in \ell_5(\Z[\Z_2]^h,(1-T)\theta)$ is elementary.
\end{proposition}

In the case of~$\Z_2$-spheres we have~$\theta=(-1)$ and Proposition~\ref{prop:NewElementaryCriterion} therefore leads us to study~$\ell_5(\Z,-2)$. 
More broadly, Proposition~\ref{prop:NewElementaryCriterion} implies that the study of the~$\ell$-monoid over~$\Z$ can be helpful to study the~$\ell$-monoid over~$\Z[\Z_2]$.
As we briefly recall, the~$\ell$-monoid over~$\Z$ is actually quite approchable.

Given a quadratic form~$v=(V,\theta)$ over $\Z$, the monoid~$\ell_5(v)$ can be recast in terms of the isometries of the boundary linking form of~$v$.
We recall the relevant terminology.
In what follows we write $\lambda=\theta+\theta^*$ for the symmetrisation of a quadratic form~$\theta$ and $\widehat{\lambda} \colon V \to V^*$ for its adjoint.
\begin{itemize}[leftmargin=*]\setlength\itemsep{0em}
\item 
The \emph{boundary split quadratic linking form} of a quadratic form~$(V,\theta)$ over $\Z$ is the split quadratic linking form~$\partial (V,\theta):=(\operatorname{coker}(\widehat{\lambda}),\partial \lambda,\partial \theta)$ with 
\begin{align*}
&\partial \lambda \colon \coker(\widehat{\lambda}) \times \coker(\widehat{\lambda})\to \Q/\Z, ([x],[y])\mapsto \frac{1}{s}y(z),  \\
& \partial \theta \colon \operatorname{coker}(\widehat{\lambda}) \to \Q/\Z, ([x]),
z \mapsto \frac{1}{s^2}\theta(z,z),
\end{align*}
where~$sx=\widehat{\lambda}(z)$ for~$z \in V$ and~$s \in \Z \setminus \lbrace 0 \rbrace$.
\item An \emph{isometry} of~$\partial (V,\theta)$ is an isomorphism~$f \colon \coker(\widehat{\lambda}) \xrightarrow{\cong} \coker(\widehat{\lambda})~$ such that
$$\partial \theta(f(x))=\partial \theta(x)$$
 for every~$x\in \coker(\widehat{\lambda})$.
The group of automorphisms of~$\partial v=\partial (V,\theta)$ is 
$$ \Aut(\partial v):=\lbrace f \colon \coker(\widehat{\lambda})  \to \coker(\widehat{\lambda}) \mid f \text{ is an isometry of } \partial v \rbrace.$$
\item 
Given an automorphism~$h \colon (V,\theta) \to (V,\theta)$, one verifies that~$(h^*)^{-1} \colon V^* \to V^*$ induces an isomorphism~$\coker(\widehat{\lambda}) \to \coker(\widehat{\lambda})$ on the cokernels.
This isomorphism, which we denote by~$\partial h:=(h^*)^{-1}$, is an isometry of the boundary split quadratic linking  forms.
This leads to a left action
$$ \Aut(V,\theta) \times \Aut(V,\theta) \curvearrowright \Aut(\partial(V,\theta)), \quad  (g,h) \cdot f=\partial h \circ f\circ \partial g^{-1}.~$$
The \emph{boundary automorphism set} is then defined as 
$$\operatorname{bAut}(v)=:\Aut(\partial v)/\Aut(v) \times \Aut(v).$$
\end{itemize}

The following theorem is due to Crowley-Sixt~\cite{CrowleySixt}; see also~\cite[Section 7.4]{ConwayOrsonPowell} for a discussion of the conventions surrounding this result.

\begin{theorem}[{\cite[Section 6.3]{CrowleySixt}}]
\label{thm:CS}
Given a nondegenerate quadratic form~$v=(V,\theta)$ over~$\Z$, there is a bijection between~$\ell_5(v)$ and~$\operatorname{bAut}(v).$ 
\end{theorem}

This calculation allows us to determine~$\ell_5(\Z,-2)$. 
\begin{proposition}
\label{prop:Triviall5(2)}
The set~$\ell_5(\Z,-2)$ is a singleton.
\end{proposition}
\begin{proof}
By Theorem~\ref{thm:CS}, it suffices to prove that~$\operatorname{bAut}(\Z,-2)$ is a singleton.
This is a direct verification: there are only two isomorphisms of~$\Z_4$, namely multiplication by~$\pm 1$,  and these both arise as the boundary of elements in~$\Aut(\Z,-2)=\lbrace \pm 1 \rbrace$.
\end{proof}

We can now prove our main result.

\begin{theorem}
\label{thm:SurfacesWithBoundaryMain}
Let~$F_0,F_1\subseteq \C P^2$ be~$\Z_2$-spheres.
Then there is an orientation-preserving homeomorphism~$G\colon \C P^2\to \C P^2$ restricting to a homeomorphism~$F_0\to F_1$.
In the case that $[F_0]=[F_1]\in H_2(\C P^2)$, the orientation-preserving homeomorphism may be chosen to be isotopic to the identity, so that $F_0$ and $F_1$ are ambiently isotopic.
\end{theorem}

\begin{proof}
For $i=0,1$, fix a tubular neighbourhood $\overline{\nu}(F_i)\cong D^2\widetilde{\times}_4 S^2$, so that we obtain fixed identifications $\alpha_i\colon \C P^2\cong X_i\cup(D^2\widetilde{\times}_4 S^2)$. 
Under these identifications, let $f \colon \partial X_0 \to \partial X_1$ be given by $\Id_{L(4,1)}$. By Proposition~\ref{prop:StablyHomeo}, there are normal 1-smoothings~$\ol{\nu}_i \colon X_i \to B$ that are normally~$(B,p)$-bordant relative to~$f$. Write~$(W,\overline{\nu})$ for such a normal~$(B,p)$-bordism.
We obtain a modified surgery obstruction~$\Theta(W,\overline{\nu}) \in \ell_5(\Z[\Z_2])$.
By Proposition~\ref{prop:SurfaceData}, the~$\Z[\Z_2]$-quadratic form~$(\Z[\Z_2],(1-T)(-1))$ is a free Wall form for~$X_i$, and hence~$\Theta(W,\overline{\nu})$ lies in the subset~$\ell_5(\Z[\Z_2],(1-T)(-1))\subseteq \ell_5(\Z[\Z_2])$.
Proposition~\ref{prop:Triviall5(2)} shows that~$\ell_5(\Z,-2)$ is trivial, and thus Proposition~\ref{prop:NewElementaryCriterion} implies that~$\Theta(W,\overline{\nu})$ is elementary.
As the obstruction is elementary, Kreck's theorem~\cite{KreckSurgeryAndDuality} implies that~$(W,\overline{\nu})$ is normal~$(B,p)$-bordant rel.~boundary to an~$s$-cobordism.
As~$\Z_2$ is finite, it is ``good'', so the~$5$-dimensional~$s$-cobordism theorem~\cite[Theorem~7.1A]{FreedmanQuinn} applies. This shows that there exists an orientation-preserving homeomorphism~$X_0 \to X_1$, restricting to $f$ on the boundary. Extend $f$ to a homeomorphism~$\widehat{f} \colon\overline{\nu}(F_0)\to \overline{\nu}(F_0)$ (we may take this extension to be the identity map under the identifications $\overline{\nu}(F_i)\cong D^2\widetilde{\times}_4 S^2$, chosen earlier). We thus obtain an orientation-preserving homeomorphism~$G\colon \C P^2\to \C P^2$ restricting to a homeomorphism~$F_0\to F_1$.

Suppose that $[F_0]=[F_1]\in H_2(X)$. Suppose $G_*$ is identity map on $H_2(\C P^2)$. Then $G$ is isotopic to the identity on $\C P^2$, by a result of Perron~\cite{Perron} and Quinn~\cite[Theorem~1.1]{Quinn} (see also~\cite[Theorem~1]{MR561244}). Moreover, as~$F_0$ and $F_1$ are homologous and $G_*$ is the identity on $H_2(\C P^2)$, we have that the induced homeomorphism $G\colon F_0\to F_1$ is orientation-preserving, so that $F_0$ is ambiently isotopic to $F_1$ (rather than ambiently isotopic to the reversed orientation version of $F_1$).

Now suppose that $G_*$ is not the identity map on $H_2(\C P^2)$. We wish to force $G_*$ to be the identity map, so that the argument above can be used. For this, we will modify the construction in the proof above, as follows. Using the fixed parametrisations established at the beginning of the proof, we have that $G_*$ is the composition
\begin{equation}
\label{eq:sequence}
\begin{tikzcd}
H_2(\overline{\nu}(F_0))\ar[r, "(\alpha_0^{-1})_*","\cong"'] 
&H_2(D^2\widetilde{\times}_4 S^2) \ar[r, "\widehat{f}_*","\cong"'] 
&H_2(D^2\widetilde{\times}_4 S^2) \ar[r, "(\alpha_1)_*","\cong"'] 
&H_2(\overline{\nu}(F_1)),
\end{tikzcd}
\end{equation}
where $\widehat{f}$ is, at present, the identity map. As the overall composition \eqref{eq:sequence} is not the identity map, it must be multiplication by~$-1$. We wish to change the overall sign. For this, go back to the beginning of this proof and set~$f=\tau$ (see e.g.~\cite[Definition~4.6]{OrsonPowellSpines} for a definition). The extension~$\widehat{f}$ of $\tau$ to the disc bundles induces the map~$-1$ on~$H_2$; see e.g.~\cite[Lemma~4.7]{OrsonPowellSpines}. This changes the overall sign of the composition \eqref{eq:sequence}, so that the new $G_*$ is the identity map on $H_2$, as desired. This completes the proof of the theorem.

\end{proof}

\appendix

\section{$\Z$-spheres in~$\C P^2$ are unique}
\label{sec:Zspheres}

A locally flat $2$-sphere $F\subseteq \C P^2$ with~$\pi_1(\C P^2 \setminus F)\cong\Z$ is called a \emph{$\Z$-sphere}.

\begin{proposition}
\label{prop:0Class}
There is a unique~$\Z$-sphere in~$\C P^2$ up to ambient isotopy.
\end{proposition}
\begin{proof}
The exterior~$X$ of a~$\Z$-sphere~$F \subseteq \C P^2$ must have~$\pi_2(X) \cong H_2(\widetilde{X})\cong \Z[\Z]$; the Euler characteristic argument needed to show this is analogous to the one in~\cite[Claim~4]{ConwayPowell}.
Since~$\partial X\cong S^1 \times S^2$ has vanishing Alexander module, and~$\Z[\Z]$-torsion~$H_2$ (see e.g.~\cite[Lemma 3.2]{ConwayPowell}), one deduces that the equivariant intersection form~$\lambda_X$ is nonsingular.
Since this form must augment to~$Q_{\C P^2}$~\cite[Lemma 5.10]{ConwayPowell}, it follows that~$\lambda_{X}$ is the rank~$1$ form given by~$\lambda_{X}(x,y)=x\overline{y}$.
Since~$\Z$-spheres are determined up to equivalence by the equivariant intersection form of their exteriors~\cite[Theorem 1.4 (1)]{ConwayPowell},  we deduce that for every pair of~$\Z$-spheres $F_0,F_1$, with respective exteriors $X_0,X_1$,  and every isometry~$\varphi \colon \lambda_{X_0} \cong \lambda_{X_1}$, there is a homeomorphism~$\Phi$ of~$\C P^2$ whose restriction to the exteriors induces~$\varphi$.
In particular,  the spheres $F_0$ and $F_1$ are equivalent.

It remains to upgrade this equivalence to an isotopy.
According to~\cite[Theorem 1.4~(2)]{ConwayPowell} any isometry~$\varphi \colon \lambda_{X_0} \cong \lambda_{X_1}$ induces an isomorphism~$\varphi_\Z \colon H_2(\C P^2) \to H_2(\C P^2)$ with the following property: the $\Z$-spheres~$F_0,F_1$ are isotopic if and only if~$\varphi_\Z=\Id$.
The isomorphism~$\varphi_\Z$ fits into the following diagram
$$
\xymatrix{
 H_2(X_0) \ar[r]^\cong \ar[d]^{\varphi \otimes_{\Z[\Z]} \Id_\Z}& H_2(\C P^2) \ar[d]^{\varphi_\Z}  \\
H_2(X_1) \ar[r]^\cong& H_2(\C P^2) 
}
$$
in which we are identifying $H_2(X_i)$ with $H_2(X_i;\Z[\Z]) \otimes_{\Z[\Z]} \Z_\varepsilon$ where $\Z_\varepsilon$ denotes $\Z$ with the~$\Z[\Z]$-module structure induced by $T \cdot n=n$ for every $n \in \Z$; see e.g.~\cite[Lemma 5.10]{ConwayPowell} with $g=0$.

In our case~$H_2(X_i;\Z[\Z])=\Z[\Z]$ and therefore any isometry~$\varphi$ is (with respect to some bases) multiplication by~$\pm t^k$ for some~$k \in \Z$.
The definition of the module structure on $\Z_\varepsilon$ implies that~$\varphi_\Z=\pm 1$.
As a consequence,  if we happen to pick an isometry~$\varphi$ with~$\varphi_\Z=-1$, then~$\varphi':=-\varphi$ will satisfy~$\varphi'_\Z=1$.
Applying~\cite[Theorem 1.4 (2)]{ConwayPowell} we deduce that~$F_0$ and~$F_1$ are isotopic.
\end{proof}

\def\MR#1{}
\bibliography{bibliosimpleCP2}

\begin{bibdiv}
\begin{biblist}

\bib{Bonahon}{article}{
      author={Bonahon, Francis},
       title={Diff\'{e}otopies des espaces lenticulaires},
        date={1983},
     journal={Topology},
      volume={22},
      number={3},
       pages={305\ndash 314},
}

\bib{BoyerUniqueness}{article}{
      author={Boyer, Steven},
       title={Simply-connected {$4$}-manifolds with a given boundary},
        date={1986},
        ISSN={0002-9947},
     journal={Trans. Amer. Math. Soc.},
      volume={298},
      number={1},
       pages={331\ndash 357},
         url={https://doi.org/10.2307/2000623},
      review={\MR{857447}},
}

\bib{BoyerRealization}{article}{
      author={Boyer, Steven},
       title={Realization of simply-connected {$4$}-manifolds with a given
  boundary},
        date={1993},
        ISSN={0010-2571},
     journal={Comment. Math. Helv.},
      volume={68},
      number={1},
       pages={20\ndash 47},
         url={https://doi.org/10.1007/BF02565808},
      review={\MR{1201200}},
}

\bib{BredonIntroduction}{book}{
      author={Bredon, Glen~E.},
       title={Introduction to compact transformation groups},
      series={Pure and Applied Mathematics, Vol. 46},
   publisher={Academic Press, New York-London},
        date={1972},
      review={\MR{0413144}},
}

\bib{ConwayOrsonPowell}{article}{
      author={Conway, Anthony},
      author={Orson, Patrick},
      author={Powell, Mark},
       title={Unknotting nonorientable surfaces},
        date={2023},
        note={\url{https://arxiv.org/abs/2306.12305}},
}

\bib{ConwayPowell}{article}{
      author={Conway, Anthony},
      author={Powell, Mark},
       title={Embedded surfaces with infinite cyclic knot group},
        date={2023},
        ISSN={1465-3060},
     journal={Geom. Topol.},
      volume={27},
      number={2},
       pages={739\ndash 821},
         url={https://doi.org/10.2140/gt.2023.27.739},
      review={\MR{4589564}},
}

\bib{CrowleySixt}{article}{
      author={Crowley, Diarmuid},
      author={Sixt, Joerg},
       title={Stably diffeomorphic manifolds and {$l_{2q+1}(\Bbb Z[\pi])$}},
        date={2011},
        ISSN={0933-7741},
     journal={Forum Math.},
      volume={23},
      number={3},
       pages={483\ndash 538},
         url={http://dx.doi.org/10.1515/FORM.2011.016},
      review={\MR{2805192 (2012e:57056)}},
}

\bib{Finashin}{article}{
      author={Finashin, Sergey},
       title={Knotting of algebraic curves in {$\mathbb{C} P^2$}},
        date={2002},
        ISSN={0040-9383},
     journal={Topology},
      volume={41},
      number={1},
       pages={47\ndash 55},
         url={https://doi.org/10.1016/S0040-9383(00)00023-9},
      review={\MR{1871240}},
}

\bib{FriedlNagelOrsonPowell}{article}{
      author={Friedl, Stefan},
      author={Nagel, Matthias},
      author={Orson, Patrick},
      author={Powell, Mark},
       title={A survey of the foundations of four-manifold theory in the
  topological category},
        date={2019},
        note={\url{https://arxiv.org/abs/1910.07372}},
}

\bib{FreedmanQuinn}{book}{
      author={Freedman, Michael~H.},
      author={Quinn, Frank},
       title={Topology of 4-manifolds},
      series={Princeton Mathematical Series},
   publisher={Princeton University Press, Princeton, NJ},
        date={1990},
      volume={39},
        ISBN={0-691-08577-3},
      review={\MR{1201584}},
}

\bib{GeskeKjuchukovaShaneson}{article}{
      author={Geske, Christian},
      author={Kjuchukova, Alexandra},
      author={Shaneson, Julius~L.},
       title={Signatures of topological branched covers},
        date={2021},
        ISSN={1073-7928},
     journal={Int. Math. Res. Not. IMRN},
      number={6},
       pages={4605\ndash 4624},
         url={https://doi.org/10.1093/imrn/rnaa184},
      review={\MR{4230406}},
}

\bib{HambletonKreckCancellation}{article}{
      author={Hambleton, Ian},
      author={Kreck, Matthias},
       title={Cancellation of hyperbolic forms and topological four-manifolds},
        date={1993},
        ISSN={0075-4102},
     journal={J. Reine Angew. Math.},
      volume={443},
       pages={21\ndash 47},
      review={\MR{1241127}},
}

\bib{MR1506962}{article}{
      author={Kampen, Egbert R.~Van},
       title={On the {F}undamental {G}roup of an {A}lgebraic {C}urve},
        date={1933},
        ISSN={0002-9327},
     journal={Amer. J. Math.},
      volume={55},
      number={1-4},
       pages={255\ndash 260},
         url={https://doi.org/10.2307/2371128},
      review={\MR{1506962}},
}

\bib{KimTwistSpinning}{article}{
      author={Kim, H.J.},
       title={Modifying surfaces in 4-manifolds by twist spinning},
        date={2006},
        ISSN={1465-3060},
     journal={Geom. Topol.},
      volume={10},
       pages={27\ndash 56},
         url={https://doi.org/10.2140/gt.2006.10.27},
      review={\MR{2207789}},
}

\bib{KirbyProblemList}{misc}{
      author={Kirby, Rob},
       title={Problems in low-dimensional topology. ({Edited} by {Rob}
  {Kirby})},
    language={English},
         how={Kazez, {William} {H}. (ed.), {Geometric} topology. 1993 {Georgia}
  international topology conference, {August} 2--13, 1993, {Athens}, {GA},
  {USA}. {Providence}, {RI}: {American} {Mathematical} {Society}. {AMS}/{IP}
  {Stud}. {Adv}. {Math}. 2(pt.2), 35-473 (1997).},
        date={1997},
        ISBN={0-8218-0654-8},
}

\bib{KasprowskiLandPowellTeichner}{article}{
      author={Kasprowski, Daniel},
      author={Land, Markus},
      author={Powell, Mark},
      author={Teichner, Peter},
       title={Stable classification of 4-manifolds with 3-manifold fundamental
  groups},
        date={2017},
        ISSN={1753-8416},
     journal={J. Topol.},
      volume={10},
      number={3},
       pages={827\ndash 881},
         url={https://doi.org/10.1112/topo.12025},
      review={\MR{3797598}},
}

\bib{KronheimerMrowka}{article}{
      author={Kronheimer, P.},
      author={Mrowka, T.},
       title={The genus of embedded surfaces in the projective plane},
        date={1994},
        ISSN={1073-2780},
     journal={Math. Res. Lett.},
      volume={1},
      number={6},
       pages={797\ndash 808},
         url={https://doi.org/10.4310/MRL.1994.v1.n6.a14},
      review={\MR{1306022}},
}

\bib{KimRubermanTopologicalTriviality}{article}{
      author={Kim, H.J.},
      author={Ruberman, D.},
       title={Topological triviality of smoothly knotted surfaces in
  4-manifolds},
        date={2008},
        ISSN={0002-9947},
     journal={Trans. Amer. Math. Soc.},
      volume={360},
      number={11},
       pages={5869\ndash 5881},
         url={https://doi.org/10.1090/S0002-9947-08-04482-6},
      review={\MR{2425695}},
}

\bib{MR561244}{inproceedings}{
      author={Kreck, M.},
       title={Isotopy classes of diffeomorphisms of {$(k-1)$}-connected
  almost-parallelizable {$2k$}-manifolds},
        date={1979},
   booktitle={Algebraic topology, {A}arhus 1978 ({P}roc. {S}ympos., {U}niv.
  {A}arhus, {A}arhus, 1978)},
      series={Lecture Notes in Math.},
      volume={763},
   publisher={Springer, Berlin},
       pages={643\ndash 663},
         url={https://mathscinet.ams.org/mathscinet-getitem?mr=561244},
      review={\MR{561244}},
}

\bib{KreckSurgeryAndDuality}{article}{
      author={Kreck, Matthias},
       title={Surgery and duality},
        date={1999},
        ISSN={0003-486X},
     journal={Ann. of Math. (2)},
      volume={149},
      number={3},
       pages={707\ndash 754},
      review={\MR{2001a:57051}},
}

\bib{LeeWilczyOdd}{article}{
      author={Lee, Ronnie},
      author={Wilczy\'{n}ski, Dariusz~M.},
       title={Locally flat {$2$}-spheres in simply connected {$4$}-manifolds},
        date={1990},
        ISSN={0010-2571},
     journal={Comment. Math. Helv.},
      volume={65},
      number={3},
       pages={388\ndash 412},
         url={https://doi.org/10.1007/BF02566615},
      review={\MR{1069816}},
}

\bib{LeeWilczy}{article}{
      author={Lee, Ronnie},
      author={Wilczy\'{n}ski, Dariusz~M.},
       title={Representing homology classes by locally flat {$2$}-spheres},
        date={1993},
        ISSN={0920-3036},
     journal={$K$-Theory},
      volume={7},
      number={4},
       pages={333\ndash 367},
         url={https://doi.org/10.1007/BF00962053},
      review={\MR{1246281}},
}

\bib{LWGenus}{article}{
      author={Lee, Ronnie},
      author={Wilczy\'{n}ski, Dariusz~M.},
       title={Representing homology classes by locally flat surfaces of minimum
  genus},
        date={1997},
        ISSN={0002-9327},
     journal={Amer. J. Math.},
      volume={119},
      number={5},
       pages={1119\ndash 1137},
  url={http://muse.jhu.edu/journals/american_journal_of_mathematics/v119/119.5lee.pdf},
      review={\MR{1473071}},
}

\bib{Miyazawa}{article}{
      author={Miyazawa, Jin},
       title={A gauge theoretic invariant of embedded surfaces in $4$-manifolds
  and exotic $p^2$-knots},
        date={2023},
        note={\url{https://arxiv.org/abs/2312.02041}},
}

\bib{OrsonPowellSpines}{article}{
      author={Orson, Patrick},
      author={Powell, Mark},
       title={Simple spines of homotopy 2-spheres are unique},
        date={2024},
        ISSN={0024-6115,1460-244X},
     journal={Proc. Lond. Math. Soc. (3)},
      volume={128},
      number={2},
       pages={Paper No. e12583, 25},
      review={\MR{4753775}},
}

\bib{Perron}{article}{
      author={Perron, B.},
       title={Pseudo-isotopies et isotopies en dimension quatre dans la
  cat\'{e}gorie topologique},
        date={1986},
        ISSN={0040-9383},
     journal={Topology},
      volume={25},
      number={4},
       pages={381\ndash 397},
         url={https://mathscinet.ams.org/mathscinet-getitem?mr=862426},
}

\bib{Quinn}{article}{
      author={Quinn, Frank},
       title={Isotopy of {$4$}-manifolds},
        date={1986},
        ISSN={0022-040X},
     journal={J. Differential Geom.},
      volume={24},
      number={3},
       pages={343\ndash 372},
         url={http://projecteuclid.org/euclid.jdg/1214440552},
      review={\MR{868975}},
}

\bib{Reiner}{article}{
      author={Reiner, Irving},
       title={Integral representations of cyclic groups of prime order},
    language={English},
        date={1957},
        ISSN={0002-9939},
     journal={Proc. Am. Math. Soc.},
      volume={8},
       pages={142\ndash 146},
}

\bib{Sunukjian}{article}{
      author={Sunukjian, Nathan~S.},
       title={Surfaces in 4-manifolds: concordance, isotopy, and surgery},
        date={2015},
     journal={Int. Math. Res. Not. IMRN},
      number={17},
       pages={7950\ndash 7978},
}

\bib{TeichnerThesis}{incollection}{
      author={Teichner, Peter},
       title={Topological 4-manifolds with finite fundamental group},
        date={1992},
        note={{U}niversity of {M}ainz},
}

\bib{TeichnerOnTheSignature}{article}{
      author={Teichner, Peter},
       title={On the signature of four-manifolds with universal covering spin},
        date={1993},
        ISSN={0025-5831},
     journal={Math. Ann.},
      volume={295},
      number={4},
       pages={745\ndash 759},
         url={https://doi.org/10.1007/BF01444915},
      review={\MR{1214960}},
}

\bib{Tristram}{article}{
      author={Tristram, Andrew},
       title={Some cobordism invariants for links},
        date={1969},
     journal={Proc. Cambridge Philos. Soc.},
      volume={66},
       pages={251\ndash 264},
      review={\MR{0248854 (40 \#2104)}},
}

\end{biblist}
\end{bibdiv}
\end{document}